\documentclass[12pt,reqno]{article}

\usepackage[usenames]{color}
\usepackage{amssymb}
\usepackage{amsmath}
\usepackage{amsthm}
\usepackage{amsfonts}
\usepackage{amscd}
\usepackage{graphicx}
\usepackage{xcolor}
\usepackage{float}

\usepackage[colorlinks=true,
linkcolor=webgreen,
filecolor=webbrown,
citecolor=webgreen]{hyperref}

\definecolor{webgreen}{rgb}{0,.5,0}
\definecolor{webbrown}{rgb}{.6,0,0}

\usepackage{fullpage}

\usepackage{graphics}
\usepackage{latexsym}
\usepackage{epsf}
\usepackage{breakurl}

\begin{document}

\theoremstyle{plain}
\newtheorem{theorem}{Theorem}
\newtheorem{corollary}[theorem]{Corollary}
\newtheorem{lemma}[theorem]{Lemma}
\newtheorem{proposition}[theorem]{Proposition}

\theoremstyle{definition}
\newtheorem{definition}[theorem]{Definition}
\newtheorem{example}[theorem]{Example}
\newtheorem{conjecture}[theorem]{Conjecture}

\theoremstyle{remark}
\newtheorem{remark}[theorem]{Remark}

\author{Jeffrey Shallit\footnote{Research supported in part by NSERC grant 2024-03725.}\\
School of Computer Science\\
University of Waterloo\\
200 University Ave. W.\\
Waterloo, ON  N2L 3G1 \\
Canada \\
\href{mailto:shallit@uwaterloo.ca}{\tt shallit@uwaterloo.ca}}

\title{Ten Squares Force an Overlap}

\maketitle

\vskip .2 in
\begin{abstract}
We prove that every concatenation of $10$ or more binary squares contains an overlap.  The bound $10$ is best possible.  In contrast, over a ternary
alphabet, there are infinitely long overlap-free words that consist
of a concatenation of squares.
\end{abstract}

\section{Introduction}

We say that a word $w$ is a {\it square\/} if
$w = xx$ for some nonempty 
word $x$, and is an {\it overlap\/} if $w = axaxa$ for
some (possibly empty) word $x$ and a single letter $a$.  For example,
{\tt murmur} is a square, and {\tt alfalfa} is an overlap.  

Suppose $w = xyz$ for some words $x,y,z$.  We call $x$ a {\it prefix\/} of
$w$, $y$ a {\it factor\/} of $w$, and $z$ a {\it suffix\/} of $w$.
For example, for $w = {\tt concatenation}$, the word
{\tt con} is a prefix, {\tt ten} is a factor, and {\tt nation} is a suffix.
We call a word {\it squarefree\/} if it has no square factor, and
{\it overlap-free\/} if it has no factor that is an overlap.

In this paper we are mostly concerned with finite words over the
alphabet $\{0,1\}$.  We define $\overline{0} = 1$ and
$\overline{1} = 0$, and extend the notation for complement to words
in the obvious way.

The theory of overlap-free binary words was developed
initially by Axel Thue \cite{Thue:1912,Berstel:1995}.
Here an important role is played by the
{\it Thue-Morse morphism\/} $\mu$, defined by $\mu(0) = 01$ and
$\mu(1) = 10$.  We call $01$ and $10$ the {\it $\mu$-blocks}.
We also write $\mu^i$ for the $i$-fold iteration of $\mu$.

In particular, a fundamental result in the area
is the factorization theorem of Restivo and Salemi \cite{Restivo&Salemi:1985a,Restivo&Salemi:1985b},
which states that every sufficiently
long finite overlap-free binary word $w$ can be written
in the form $x \mu(y) z$, where $y$ is overlap-free, and $x,z$ are 
words of length at most two.  

In this paper we prove the following result.
\begin{theorem}
Every concatenation of $10$ or more binary squares contains a factor
that is an overlap, and the bound $10$ is optimal.
\label{main}
\end{theorem}

The rough idea of the proof is as follows:  we assume, to get
a contradiction, there is a shortest concatenation of
$10$ squares $w$ that is overlap-free.  
If the squares in $w$ are aligned with the $\mu$-blocks of the
Restivo-Salemi factorization, then we can use
that factorization to construct a shorter overlap-free concatenation of
$10$ squares.
If the squares are not aligned with the $\mu$-blocks,
their structure is constrained,
and we can simply use breadth-first search to show that $9$ squares
is the maximum possible.

All words in this paper are indexed starting at position $0$.

\section{Useful lemmas}

In this section we state some of the basic results needed for
the proof.

\begin{lemma} 
Let $x,y \in \{0,1\}^*$ and $a \in \{0,1\}$.
\begin{itemize}
\item[(a)] $x$ is overlap-free if and only if $\mu(x)$ is overlap-free.
\item[(b)] If $a\, x$ contains an overlap, so does
$\overline{a}\, \mu(x)$.
\item[(c)] If $xa$ contains an overlap, so does $\mu(x) a$ .
\item[(d)] If $x^2 = \mu(y)$ then $x = \mu(t)$ for some word $t$ and
$y = tt$.
\end{itemize}
\label{four}
\end{lemma}

\begin{proof}
\leavevmode
\begin{itemize}
\item[(a)] See \cite[Satz 6]{Thue:1912} or \cite[Lemma 1.7.4]{Allouche&Shallit:2003}.
\item[(b)]   Suppose $a\, x$ contains an overlap.  Then either
the overlap is entirely within $x$, or $yaya$ is a prefix of $x$ for some $y$.
In the former case, $\mu(x)$ contains an overlap by (a), and hence
so does $\overline{a} \mu(x)$.  In the latter case, 
$\overline{a} \, \mu(x) = \overline{a} \mu(yaya) =
\overline{a} \mu(y) a \overline{a} \mu(y) a \overline{a}$, which clearly
contains the overlap $\overline{a} z \overline{a} z \overline{a}$
with $z = \mu(y) a$.
\item[(c)]  Analogous to (b).
\item[(d)]  See, for example, \cite[Lemma 1.7.3]{Allouche&Shallit:2003}
\end{itemize}
\end{proof}

\begin{lemma}[Restivo-Salemi factorization]
Let $w \in \{0,1\}^*$ be an overlap-free finite word.  Then there exist
$u, v \in \{\epsilon, 0, 1, 00, 11 \}$ and an overlap-free 
word $y$ such that $w = u \mu(y) v$.  If $|w| \geq 7$ the factorization
is unique.
\label{restivo}
\end{lemma}

\begin{proof}
The theorem is essentially in \cite{Restivo&Salemi:1985a,Restivo&Salemi:1985b}.  The version we use here can essentially be found in
\cite{Kfoury:1988b} or \cite[Prop.~1.7.5]{Allouche&Shallit:2003}.
\end{proof}

\begin{lemma}
Let ${\cal A} = \{ 0^2, 1^2, (010)^2, (101)^2 \}$.   Then a word
$w \in \{0,1\}^*$ is an overlap-free square if and only if $w$ is
a conjugate of $\mu^i (x)$ for $x\in \cal A$ and $i \geq 0$.
In particular, an overlap-free square has length either
$2^i$ or $3 \cdot 2^i$ for $i \geq 1$.
\label{ovfl2}
\end{lemma}

\begin{proof}
See Thue \cite{Thue:1912} or Shelton and Soni \cite{Shelton&Soni:1985}.
\end{proof}

%\begin{lemma}
%Suppose $y, y', u \in \{0,1\}^*$ and $c,d \in \{0,1\}$ are such that
%$u = c \mu(y) = \mu(y') d$.  Then $u = c(\overline{c} c)^{|y|}$.
%\end{lemma}
%
%\begin{proof}
%See \cite[Lemma 1.7.2]{Allouche&Shallit:2003}.
%\end{proof}

\begin{lemma}
Suppose  $a, b \in \{0,1 \}$, $x,y \in \{0,1\}^+$, and $y$ is overlap-free.
If $ax^2b = \mu(y)$, then $x \in \{0,1,010,101\}$.
\label{ax2b}
\end{lemma}

\begin{proof}
If $x$ is of even length, say $|x| = 2m > 0$,
then $|\mu(y)| = 4m+2$ and so
$|y| = 2m+1$.  Write $y = c_0 \cdots c_{2m}$; then
$\mu(y) = c_0 \overline{c_0} \cdots c_{2m}\overline{c_{2m}}$.
Thus 
$$x = \overline{c_0} c_1 \cdots c_{m-1} \overline{c_{m-1}} c_{m} =
\overline{c_{m}} c_{m+1} \cdots \overline{c_{2m-1}} c_{2m}.$$
It follows that $\overline{c_i} = \overline{c_{m+i}}$
for $0 \leq i \leq m-1$ and $c_i = c_{m+i}$ for $1 \leq i \leq m$;
hence $c_i = c_{m+i}$ for $0 \leq i \leq m$.
Thus $y = c_0\cdots c_{m-1} c_0 \cdots c_{m-1} c_0$, an overlap, which is
a contradiction.

Hence $x$ is of odd length.  Since $y$ is overlap-free, by
Lemma~\ref{four} (a), so is $\mu(y) = ax^2b$, and hence so is $x^2$.
By Lemma~\ref{ovfl2}, the only possibilities for $x$
are $x \in \{0,1,001,010,011,100,101,110 \}$.
Of these the only ones for which $\mu(y) = ax^2b$ has a solution
are $\{0,1,010,101\}$.
\end{proof}

We define two functions on binary words.  We define
$L(w)$ to be $1$ if there exists $a \in \{0,1\}$ such that
$aw$ is overlap-free, and $0$ otherwise.  Similarly, we
define $R(w)$ to be $1$ if there exists $a \in \{0,1\}$ such that
$wa$ is overlap-free, and $0$ otherwise.  Thus $L$ measures whether
an overlap-free word can be extended one symbol to the left to get
a longer overlap-free word, and
$R$ measures the same thing for extension to the right.

\begin{lemma}
Let $w \in \{0,1\}^*$ be overlap-free with $|w| \geq 16$.
Let its Restivo-Salemi factorization be $w = u \mu(y) v$, with
$y$ overlap-free, and $u, v \in \{\epsilon, 0, 1, 00, 11 \}$.
\begin{itemize}
\item[(i)] If $u = aa$ for some $a \in \{0,1 \}$, then $L(w) = 0$.
\item[(ii)] If $v = aa$ for some $a \in \{0,1\}$, then $R(w) = 0$.
\end{itemize}
\label{tech}
\end{lemma}

\begin{proof}
We prove only (i), as (ii) follows with a symmetric argument.
Since $|w| \geq 16$ and $|v| \leq 2$, it follows that $|y| \geq 6$.
Checking all eight $3$-symbol prefixes of $y$, we easily confirm
that $aa \mu(y)$ being overlap-free implies that
$y$ begins with $\overline{a} a \overline{a}$. Thus $w$ begins
with $a a \mu(\overline{a} a \overline{a})$.  If we prepend
$a$, then the resulting word begins with the overlap $aaa$.
If we prepend $\overline{a}$ then the resulting word
begins $\overline{a} a a \overline{a} a a \overline{a}$, which is
an overlap.  Hence no left extension is overlap-free, and $L(w) = 0$.
\end{proof}

\begin{lemma}
Define
$$ A = 0^2, \quad B = 1^2, \quad C = (010)^2, \quad D = (101)^2.$$
Every concatenation of nine or more words chosen from
${\cal A} = \{A,B,C,D\}$ contains an overlap.  Furthermore, no product of
$8$ words is extendable to the left or to the right and still
be overlap-free.
\label{abcd2}
\end{lemma}

\begin{proof}
We use breadth-first search in the tree of all finite words
that are products of $A, B, C, D$, and discarding those that have an overlap.
The results are depicted in Table~\ref{abcd}.
\begin{table}[H]
\begin{center}
\begin{tabular}{c|l}
$n = 1$ & $\{ A , B , C , D \}$ \\
$n = 2$ & $\{AB, AD, BA, BC, CB, DA \}$\\
$n = 3$ & $\{ ABA, ABC, ADA, BAB, BAD, BCB, CBA, CBC, DAB, DAD \}$ \\
$ n = 4$ & $\{ABAB, ABAD, ABCB, ADAB, ADAD, BABA, BABC, $ \\
 & \quad\quad  $ BADA, BCBA, BCBC, CBAB, CBCB, DABA, DADA \}$ \\
$n = 5$ & $\{ ABADA, ABCBA, ADABA, BABCB, BADAB, BCBAB, CBABC, DABAD\}$ \\
$n = 6$ & $\{ ABADAB, ABCBAB, ADABAD, BABCBA, BADABA, $ \\
 & \quad\quad $ BCBABC, CBABCB, DABADA \}$ \\
$n = 7$ & $\{ ABADABA, ABCBABC, ADABADA, BABCBAB, BADABAD, $ \\
 & \quad\quad $ BCBABCB, CBABCBA, DABADAB \}$ \\
$n = 8$ &  $\{ABADABAD, ABCBABCB, ADABADAB, BABCBABC, BADABADA, $ \\
 & \quad\quad $ BCBABCBA, CBABCBAB, DABADABA \}$ \\
$n = 9$ & $\emptyset$
\end{tabular}
\end{center}
\caption{Overlap-free words that are formed from blocks $A, B, C, D$
for lengths $1$ through $9$.}
\label{abcd}
\end{table}
At length $n = 9$ no words survive, which proves the first claim.

For the second claim, we simply check the $8$ possible words of length
$8$ and verify that they have no overlap-free extension on either side.
\end{proof}

\section{Proof of the main result}

\begin{theorem}
Let $w$ be an overlap-free binary word, and suppose
$w = S_1 S_2 \cdots S_k$, where each $S_i$ is a square.
Then $k + L(w) + R(w) \leq 9$.
\label{main2}
\end{theorem}

\begin{proof}
Suppose the result is false, and let $w$ be such that
$k + L(w) + R(w) \geq 10$, and $|w|$ is as small as possible.
This means that $k \geq 8$, and hence $|w| \geq 16$.  By
Lemma~\ref{restivo}, the Restivo-Salemi factorization of
$w$ is unique, and we can write
$w = u \mu(y) v$, where $u, v \in \{\epsilon, 0, 1, 00, 11 \}$.
Because $|w|$ is even, we know that $|u|$ and $|v|$ have the same parity.

\bigskip

\noindent{\it Case 1:}  $|u|$ and $|v|$ are both even.  Thus
$u, v \in \{ \epsilon, 00, 11 \}$.
If $u \not= \epsilon$, set $c = 1$; otherwise set $c = 0$.
If $v \not= \epsilon$, set $d = 1$; otherwise set $ d= 0$.
Define $w' = S_{c+1} S_{c+2} \cdots S_{k-d}$; the effect is to remove
the leftmost square from $w$ if $u$ is nonempty, and similarly for
the rightmost square if $v$ is nonempty.  

Clearly all the squares in $w$ start at even positions, and since
$|u|$ is even, all the images of individual letters of $y$ under $\mu$
also start at even positions.  Thus each square in $w'$
lies inside $\mu(y)$ and begins at an even position.   Hence
each square $S_i$ can be written as $S_i = \mu(T_i)$ for some 
nonempty word $T_i$.  By Lemma~\ref{four} (d)
we know that $T_i$ is itself a square.
Thus $w' = \mu(z)$, where $z = T_{c+1} T_{c+2} \cdots T_{k-d}$ is
an overlap-free product of $k-c-d$ nonempty squares, and $|z| < |w|$.

We now argue that $L(z) \geq L(w) + c$.
There are three subcases to consider, based on the values of $c$
and $L(w)$.

\bigskip

\noindent{\it Case (a): $c = 1$.}
Then $u = aa$ for some $a \in \{0,1\}$, and then Lemma~\ref{tech}
shows $L(w) = 0$.  The word $w'$ has some letter $b \in \{0,1\}$ 
immediately to its left in $w$, and $bw' = b \mu(z)$ is overlap-free.
By the contrapositive of Lemma~\ref{four} (b),
we know that $\overline{b}z$ is overlap-free, so
$L(z) = 1$.  Hence $L(z) \geq L(w) + c$.  

\bigskip

\noindent{\it Case (b): $c = 0$ and $L(w) = 1$.}   Then $bw$ is
overlap-free for some $b \in \{0,1\}$.  Then $bw' = b \mu(z)$ is a factor
of $bw$, and so the contrapositive of Lemma~\ref{four} (b)
again gives $L(z) = 1$.  Hence
$L(z) \geq L(w) + c$ again.

\bigskip

\noindent{\it Case (c):  $c = 0$ and $L(w) = 0$.}  There is nothing
more to prove in this case.

\bigskip

\noindent Thus, in all three cases, we have $L(z) \geq L(w) + c$.

The case of the right side is exactly parallel, and using the same
reasoning we get $R(z) \geq R(w) + d$.

Combining the two inequalities, we get
$(k - c - d) + L(z) + R(z) \geq k + L(w) + R(w) \geq 10$.
Thus $z$ is a shorter counterexample, contradicting our assumption that
$w$ was of minimal length.

\bigskip

\noindent {\it Case 2:  $|u|=|v| = 1$.}  Let $c' = 1-L(w)$ and
$d' = 1-R(w)$.  Define
\begin{equation}
w'' = S_{c'+1} S_{c'+2} \cdots S_{k-d'}. \label{foo5}
\end{equation}
Then $w''$ is the product of 
$m = k-(1-L(w))-(1-R(w)) = k+L(w) + R(w) - 2 \geq 8$
squares.  Furthermore, the starting position
of each square is offset by $1$ with respect to the $\mu$-blocks
in the Restivo-Salemi factorization of $w$.

We claim that each of the square factors in the product 
\eqref{foo5}
is either $0^2, 1^2, (010)^2$, or $(101)^2$.  Let $x^2$ be one of the
square factors in $w''$.
Then $x^2$ has a possible overlap-free extension on both
the left and right:  if $x^2$ is not the first square in $w$, then
we can use its left neighbor in $w$.  Otherwise it is the first
square in $w$, and hence has a left extension to an overlap-free word
by the definition of $c'$.  The analogous statement holds on the
right side.  Choose $a,b \in \{0,1\}$ such that $ax^2$ and $x^2b$
are both overlap-free.

Let $x = x_0 x_2 \cdots x_{n-1}$.  Since $ax^2$ is overlap-free, we must have
$a \not=x_{n-1}$.  Hence $a = \overline{x_{n-1}}$.  Similarly, since
$x^2b$ is overlap-free, we see that $b = \overline{x_0}$. 

Now use the fact that $x^2$ has at least one neighbor in the product $w''$.
If it has a right neighbor, then the letter $b$ following $x^2$ lies
inside $w$, and therefore $x_{n-1} b$ is a $\mu$-block.
Hence $b = \overline{x_{n-1}}$.   Thus $x_0 = x_{n-1}$.
The case of a left neighbor is symmetric.
Hence $a = b$ and $ax^2b$ is a $\mu$-image.

It now follows from the one-symbol offset between the square-factorization 
\eqref{foo5} and the Restivo-Salemi factorization 
that $x = \overline{a} x' \overline{a}$, and hence
$$ax^2b = \underbrace{a \overline{a}} \ \underbrace{x' \overline{a}} \ 
\underbrace{\overline{a} x'} \ \underbrace{\overline{a} a} \ ,$$
where each underbrace delineates a $\mu$-image.  Thus
$a x^2 b = \mu(t)$ for some overlap-free $t$.  By Lemma~\ref{ax2b},
it follows that
\eqref{foo5} is a product of $m$ words chosen from $\cal A$.
By Lemma~\ref{abcd2},
we know $m < 9$, so $m = 8$.  But no overlap-free product of $8$
words from $\cal A$ has a one-letter overlap-free extension on
either side.
This contradicts the second
assertion of Lemma~\ref{abcd2}.

This completes the proof.
\end{proof}

We are now ready to complete the proof of Theorem~\ref{main}.
\begin{proof}
If there were an overlap-free
concatenation of ten or more squares, 
then Theorem~\ref{main2} would give
$10 \leq k \leq k + L(w) + R(w) \leq 9$, a contradiction.

The bound $10$ is best possible, as it is easy to check that
$$ (001)^2 (10)^2 (01)^2 (10)^2 (011001)^2 (10)^2 (01)^2 (10)^2 (011)^2$$
is an overlap-free product of nine squares.
\end{proof}

\section{Larger exponents}

Karhum\"aki and Shallit \cite{Karhumaki&Shallit:2004} proved that 
there is a factorization theorem for binary
words analogous to that of Restivo-Salemi,
for all exponents $e$ with $2< e \leq 7/3$.  So it is natural to wonder
if Theorem~\ref{main} has any counterpart for larger $e$.
The answer is no, since Currie and Rampersad
proved that for each real $e > 2$ there is an infinite
$e$-power-free word that contains
squares beginning
at every position \cite[Theorem 3.4]{Currie&Rampersad:2010}.

\section{Larger alphabets}

\begin{proposition}
Over a ternary alphabet,
there are uncountably many overlap-free infinite words
that are the concatenation of square factors.
\end{proposition}

\begin{proof}
It is well-known that
there are uncountably many infinite squarefree
words $\bf w$
over the alphabet $\{0,1,2\}$.
For example, see \cite[Theorem 1.8]{Bean&Ehrenfeucht&McNulty:1979}.
Define a morphism $h$ as follows:
$h(a) = aa$ for $a \in \{0,1,2\}$.
Clearly ${\bf z} := h({\bf w})$ is a concatenation
of squares.  We claim that $\bf z$ is overlap-free.

Suppose, contrary to what we want to prove, that ${\bf z}$ has
an overlap $axaxa$.  Clearly $\bf z$ cannot contain $aaa$, so
assume $|x| \geq 1$.

\bigskip

\noindent{\it Case 1:}  $|ax|$ is even.
If the starting position of $axaxa$ in $\bf z$ is even, it is
aligned with the $h$-blocks in $\bf z$.
Then there exists a nonempty word $t$ such that $h(t) = ax$ and
$tt$ is a factor of $\bf w$, a contradiction.

If the starting position of $axaxa$ is odd, then it is offset
by $1$ from the $h$-blocks in $\bf z$.
Then the same argument applies with $h(t) = xa$.

\bigskip

\noindent{\it Case 2:}  $|ax|$ is odd.  
If the starting position of $axaxa$ in $\bf z$ is even,
then the starting position of the second $axa$ is odd.
Since the effect of $h$ is to double letters, this forces
the last letter of $x$ to be $a$, and the first letter of
$x$ to be $a$.  Thus $\bf z$ contains $aaa$, a contradiction.

If the starting position of $axaxa$ in $\bf z$ is odd, then
the symbol in one position to the left must also be $a$.
Then the second $axa$ starts at an even position, which
forces the first letter of $x$ to be $a$.   Thus again $\bf z$
contains $aaa$, a contradiction.

\bigskip

\noindent Hence the desired result follows.
\end{proof}

\section*{Acknowledgments}

Theorem~\ref{main} was conjectured by the author on November 29 2023.
The general idea of how to prove it was clear at that
time, but some details remained murky.
Some suggestions were obtained during an extended discussion with
ChatGPT 5.5 Pro in May 2026.
The arguments it suggested have been checked, re-organized,
and rewritten by the author.


\begin{thebibliography}{10}

\bibitem{Allouche&Shallit:2003}
J.-P. Allouche and J.~Shallit.
\newblock {\em Automatic Sequences: {T}heory, Applications, Generalizations}.
\newblock Cambridge University Press, Cambridge, 2003.

\bibitem{Bean&Ehrenfeucht&McNulty:1979}
D.~A. Bean, A.~Ehrenfeucht, and G.~McNulty.
\newblock Avoidable patterns in strings of symbols.
\newblock {\em Pacific J. Math.} {\bf 85} (1979), 261--294.

\bibitem{Berstel:1995}
J.~Berstel.
\newblock {\em Axel {Thue's} Papers on Repetitions in Words: a Translation}.
\newblock Number~20 in Publications du Laboratoire de Combinatoire et
  d'Informatique {Math\'ematique}. Universit\'e du Qu\'ebec \`a Montr\'eal,
  February 1995.

\bibitem{Currie&Rampersad:2010}
J.~Currie and N.~Rampersad.
\newblock Infinite words containing squares at every position.
\newblock {\em RAIRO Inform. Th\'eor. App.} {\bf 44} (2010), 113--124.

\bibitem{Karhumaki&Shallit:2004}
J.~Karhum{\"{a}}ki and J.~Shallit.
\newblock Polynomial versus exponential growth in repetition-free binary words.
\newblock {\em J. Combin. Theory. Ser. A} {\bf 105}(2) (2004), 335--347.

\bibitem{Kfoury:1988b}
A.-J. Kfoury.
\newblock Square-free and overlap-free words.
\newblock In G.~Mirkowska and H.~Rasiowa, editors, {\em Mathematical Problems
  in Computation Theory}, Vol.~21 of {\em Banach Center Publications}, pp.
  285--297. PWN -- Polish Scientific Publishers, Warsaw, 1988.

\bibitem{Restivo&Salemi:1985a}
A.~Restivo and S.~Salemi.
\newblock Overlap free words on two symbols.
\newblock In M.~Nivat and D.~Perrin, editors, {\em Automata on Infinite Words},
  Vol. 192 of {\em Lecture Notes in Computer Science}, pp.  198--206.
  Springer-Verlag, 1985.

\bibitem{Restivo&Salemi:1985b}
A.~Restivo and S.~Salemi.
\newblock Some decision results on nonrepetitive words.
\newblock In A.~Apostolico and Z.~Galil, editors, {\em Combinatorial Algorithms
  on Words}, pp.  289--295. Springer-Verlag, 1985.

\bibitem{Shelton&Soni:1985}
R.~O. Shelton and R.~P. Soni.
\newblock Chains and fixing blocks in irreducible binary sequences.
\newblock {\em Discrete Math.} {\bf 54} (1985), 93--99.

\bibitem{Thue:1912}
A.~Thue.
\newblock {\"Uber} die gegenseitige {Lage} gleicher {Teile} gewisser
  {Zeichenreihen}.
\newblock {\em Norske vid. Selsk. Skr. Mat. Nat. Kl.} {\bf 1} (1912), 1--67.
\newblock Reprinted in {\it Selected Mathematical Papers of Axel Thue}, T.
  Nagell, editor, Universitetsforlaget, Oslo, 1977, pp.~413--478.

\end{thebibliography}
\end{document}